\documentclass[a4paper, 11pt, reqno, final]{amsart}

\usepackage{amsmath,amssymb,amsthm}
\usepackage[mathscr]{eucal}

\numberwithin{equation}{section}
\allowdisplaybreaks

\theoremstyle{definition}
 \newtheorem{thm}{Theorem}[section]
 \newtheorem{prp}[thm]{Proposition}
 \newtheorem{lem}[thm]{Lemma}
 
 \newtheorem{dfn}[thm]{Definition}
 
 \newtheorem{rmk}[thm]{Remark}
 \newtheorem{eg} [thm]{Example}
 \newtheorem{cnj}[thm]{Conjecture}
 \newtheorem*{ack}{Acknowledgements}
 \newtheorem*{rmk*}{Remark}
 \newtheorem*{not*}{Notation}
 \newtheorem*{thm*}{Theorem}

\newcommand{\bbC}{\mathbb{C}}
\newcommand{\bbP}{\mathbb{P}}

\newcommand{\bbZ}{\mathbb{Z}}

\newcommand{\calP}{\mathcal{P}}

\newcommand{\seteq}{\mathbin{:=}}
\newcommand{\vac}[1]{\left|#1 \right>}
\newcommand{\covac}[1]{\left< #1 \right|}
\newcommand{\pair}[2]{\left< #1 \kern 0.05em{|}   #2 \right>}
\makeatletter
\def\ch{\mathop{\operator@font ch}\nolimits}
\def\id{\mathop{\operator@font id}\nolimits}
\def\sgn{\mathop{\operator@font sgn}\nolimits}
\def\Aut{\mathop{\operator@font Aut}\nolimits}
\def\End{\mathop{\operator@font End}\nolimits}
\def\Ind{\mathop{\operator@font Ind}\nolimits}
\def\Mod{\mathop{\operator@font Mod}\nolimits}
\def\Res{\mathop{\operator@font Res}}
\def\SU{\mathop{\operator@font SU}\nolimits}
\def\Vir{\mathop{\operator@font Vir}\nolimits}
\makeatother

\title[5d AGT conjecture]{Five-dimensional SU(2) AGT conjecture and recursive formula of deformed Gaiotto state}
\date{May 21, 2010; revised December 11, 2010}
\author{Shintarou Yanagida}
\address{Kobe University, Department of Mathematics, Rokko, Kobe 657-8501, Japan}
\email{yanagida@math.kobe-u.ac.jp}

\begin{document}

\begin{abstract}
This paper deals with the five-dimensional pure $\SU(2)$ Alday-Gaiotto-Tachikawa (AGT) conjecture 
proposed by Awata and Yamada.
We give a conjecture on a recursive formula for the inner product 
of the deformed Gaiotto state.
We also show that the $K$-theoretic pure $\mathrm{SU}(2)$ Nekrasov 
partition function satisfies the same recursion relation.
Therefore the five-dimensional AGT conjecture is 
reduced to our conjectural recursive formula.
\end{abstract}

\maketitle

\section{Introduction}

In \cite{AY:2010}, Awata and Yamada proposed a conjecture which relates 
the instanton part of Nekrasov's five-dimensional pure $\SU(2)$ partition 
function $Z(\Lambda,Q,q,t)$ \cite{N:2003,NY:2005} 
to the deformed Virasoro algebra \cite{SKAO:1996}.
The conjecture claims that $Z(\Lambda,Q,q,t)$ coincides with 
the inner product $\pair{G}{G}$ of the so-called deformed Gaiotto 
state $\vac{G}$. 
This state is a natural $q$-deformed analogue of the state proposed by Gaiotto 
in \cite{G:2008}.
These states are sorts of Whittaker vectors \cite{K:1978}.

This conjecture is a natural extension of the degenerate version \cite{G:2008} 
of the four-dimensional AGT conjecture \cite{AGT:2010}.
The original AGT conjecture claims that the instanton part of 
Nekrasov's four-dimensional $\SU(2)$ partition function with $N_f=4$ 
anti-fundamental hypermultiplets \cite{N:2003} coincides with 
the conformal block of the Virasoro algebra.

In \cite{FL:2010}, a strategy was proposed for the proof 
of the four-dimensional conjecture 
in the case of the theory with adjoint matter multiplet.
Its main idea is to show that both the Nekrasov partition function and 
the one-point conformal  block on the torus satisfy the same recursion formula.
On the conformal field theory side the recursion formula was conjectured 
by Poghossian \cite{P:2009} and derived in \cite{HJS:2010a}. 
On the Nekrasov side Fateev and Litvinov \cite{FL:2010} used 
the integral expression of the four-dimensional partition function to 
analyze its poles and residues, and they obtained the same recursion formula.

In \cite{HJS:2010b}, the same strategy was used for the proof of 
four-dimensional conjectures with $N_f=0,1,2$ anti-fundamental 
hypermultiplets. 
In these cases, 
the Nekrasov partition functions were again analytically analyzed 
with the help of the integral expressions 
which is similar to that of \cite{FL:2010}.
And the conforma field theoretical recursion formulas were derived through 
degeneration (decoupling) limit technique \cite{MMM:2009} 
of the Zamolodchikov elliptic recursive formula for the four-point 
conformal block on the sphere \cite{Z:1984}.

Note that the strategy of the proofs of the AGT conjectures 
using recursive formulas are restricted to the case of $N_f\le 2$. 
This is because there are some difficulties in the asymptotic analysis 
on the integral expression 
of the Nekrasov partition function for the theory with $N_f\ge3$, 
as mentioned in \cite{FL:2010, HJS:2010b}. 

What we may stress here is that 
in our consideration and in those proofs 
only the instanton part of the partition function is considered, 
although the full partition function consists of three contributions : 
classical, instanton and one-loop parts.
As mentioned in \cite[\S 4]{AGT:2010}, 
the one-loop factor reproduces 
the product of the DOZZ (Dorn-Otto-Zamolodchikov-Zamolodchikov) 
three-point functions \cite{DO:1994,ZZ:1996} 
of the Liouville theory.
And there is a general proposal that 
Nekrasov's full partition function, integrated 
over the vev (vacuum expectation value)
with a natural measure, is the $n+3$ point correlation function on a sphere, 
and is also related to the $n$ point function on a torus.
This point of view is out of the scope of our paper.

The aim of this paper is two-fold.
First we propose Conjecture \ref{cnj:main} on the recursion formula for the inner product $\pair{G}{G}$ of the deformed Gaiotto state.
This is an analogue of the recursive formulas appearing 
in \cite{Z:1984,P:2009,MMM:2009,HJS:2010a}.
Our conjecture suggests a new resemblance between 
the Fock representation of the Virasoro algebra 
and that of the deformed Virasoro algebra.

The second aim is to prove Theorem \ref{thm:main} which states that 
the $K$-theoretic Nekrasov partition function $Z(\Lambda,Q,q,t)$ satisfies 
the same recursion formula proposed in Conjecture \ref{cnj:main}.
In the proof we utilize the integral form of the five-dimensional 
partition function, which in itself seems to be new.

As a consequence, the conjecture of Awata and Yamada \cite{AY:2010} 
is reduced to Conjecture \ref{cnj:main}.
At this point, not enough is known on the representation of the deformed 
Virasoro algebra, so that we do not have a clue to solve our conjecture.

This paper is organized as follows.
In \S \ref{sect:dva} we recall the deformed Virasoro algebra and 
the deformed Gaiotto state.
We will state the main conjecture on the inner product of the deformed Gaiotto 
state in this section.
In \S \ref{sect:Nek} we introduce the integral expression of the $K$-theoretic 
Nekrasov partition function $Z(\Lambda,Q,q,t)$.
In \S \ref{sect:rec} we prove Theorem \ref{thm:main}. 

\begin{not*}
Throughout in this paper, we follow \cite{M:1995:book} for the notations 
of partitions. 
For a positive integer $n$, a partition $\lambda$ of $n$ is a series of 
positive integers $\lambda=(\lambda_1,\lambda_2,\ldots)$ such that 
$\lambda_1\ge\lambda_2\ge\cdots$ and $\lambda_1+\lambda_2+\cdots=\lambda$.
If $\lambda$ is a partition of $n$, then we define $|\lambda|\seteq n$.
The number $\ell(\lambda)$ is defined to be 
the length of the sequence $\lambda$.
The conjugate partition of $\lambda$ is denoted by $\lambda'$.

In addition we denote by $\calP$ 
the set of all the partitions of natural numbers 
including the empty partition $\emptyset$.
we also denote by $\calP_n$ the set of partitions of $n$.
$p(n)\seteq\#\calP_n$ denotes the number of partitions of $n$.

We also follow \cite{M:1995:book} for the convention of the Young diagram.
Thus the first coordinate $i$ (the row index) increases as one goes downwards,
and the second coordinate $j$ (the column index) increases as one goes 
rightwards. 
We denote by $\square=(a,b)$ the box located at the coordinate $(a,b)$,
and denote the coordinate by $i(\square)\seteq a$ and $j(\square)\seteq b$.

For a box $\square=(i,j)$, the (relative) arm and the leg are defined to be 
\begin{align*}
a_{\lambda}(\square)\seteq \lambda_i-j,\quad
l_{\lambda}(\square)\seteq \lambda'_j-i.
\end{align*}

Finally, all the algebras are defined over $\bbC$.
\end{not*}

\begin{ack}
The author is supported by JSPS Fellowships for Young Scientists (No.21-2241).
He would like to thank Prof. Yasuhiko Yamada for the valuable discussion, 
and also thank the adviser Prof. K\={o}ta Yoshioka for his interest 
and comments. 
\end{ack}

\section{Deformed Virasoro algebra $\Vir_{q,t}$}
\label{sect:dva}

\subsection{Definition}
Let $q,t$ be two generic complex parameters.
Set $p\seteq q/t$ and assume that $p$ is not a root of $-1$.
The deformed Virasoro algebra $\Vir_{q,t}$ \cite{SKAO:1996,BP:1998} 
is defined to be 
the associative algebra generated by $\{T_n \mid n\in\bbZ\}$ 
with relations 
\begin{align*}
[T_n,T_m]
=-\sum_{l=1}^\infty &f_l(T_{n-l}T_{m+l}-T_{m-l}T_{n+l})\\
                    &-\dfrac{(1-q)(1-t^{-1})}{1-p}(p^n-p^{-n})\delta_{m+n,0},\\
\end{align*}
where the coefficients $f_l$'s are determined through
\begin{align*}
f(x)=\sum_{l\ge 0}f_l x^l
    =\exp\Big(\sum_{n\ge1}\dfrac{(1-q)(1-t^{-1})}{1+p^n}\dfrac{x^n}{n}\Big).
\end{align*}

$\Vir_{q,t}$ has a $\bbC$-linear anti-involution $\iota$ defined by
\begin{align}
\label{eq:inv}
\iota(T_n)=T_{-n}.
\end{align}
For a partition $\lambda=(\lambda_1,\lambda_2,\ldots,\lambda_\ell)$ 
($\lambda_1\ge\lambda_2\ge\cdots>0$), 
let us introduce the symbols 
$T_\lambda\seteq T_{\lambda_\ell} T_{\lambda_{\ell-1}}\cdots T_{\lambda_1}$
and 
$T_{-\lambda}\seteq 
 \iota(T_\lambda)=T_{-\lambda_1}T_{-\lambda_2}\cdots T_{\lambda_\ell}$.

\subsection{Verma module and Deformed Gaiotto state}
\label{subsec:dGs}

This subsection follows \cite[\S 3.2]{AY:2010}.

For $h\in\bbC$, 
let $\vac{h}$ be a vector and define the actions of the generators by 
$T_0\vac{h}= h\vac{h}$ and
$T_n\vac{h}= 0$ for any $n\in\bbZ_{>0}$.
Then the Verma module $M_h$ for $\Vir_{q,t}$ is defined to be 
the $\Vir_{q,t}$-module generated by $\vac{h}$.

Let us introduce the outer grading operator $d$ satisfying $[d,T_n]=-n T_n$.
Defining the action of $d$ on $M_h$ by $d \vac{h}=0$, 
we have the direct decomposition $M_h=\oplus_{n\in\bbZ_{\ge0}}M_{h,n}$.
$M_{h,n}$ has a basis consisting of the vectors
\begin{align*}
T_{-\lambda}\vac{h}
= T_{-\lambda_1}T_{-\lambda_2}\cdots T_{-\lambda_\ell}\vac{h},
\end{align*}
each of which is indexed by a partition 
$\lambda=(\lambda_1,\lambda_2,\ldots,\lambda_\ell)$,
$\lambda_1\ge\lambda_2\ge\cdots>0$, with $|\lambda|=n$.

The dual module $M_h^*$ is similarly defined.
It is generated by the highest weight vector $\covac{h}$ satisfying 
$\covac{h}T_n=0$ for any $n\in\bbZ_{<0}$ and $\covac{h}T_0=h\covac{h}$.
$M_h^*$ has a basis consisting of the vectors
\begin{align*}
\covac{h}T_{\lambda}
=\covac{h}T_{\lambda_\ell} T_{\lambda_{\ell-1}}\cdots T_{\lambda_1},
\end{align*}
each of which is indexed by a partition 
$\lambda=(\lambda_1,\lambda_2,\ldots,\lambda_\ell)$.

Also recall the bilinear contravariant form $S$ on $M_h\otimes M_h\to\bbC$:
\begin{align*}
S(T_\lambda\vac{h},T_\mu\vac{h})\seteq 
\covac{h}T_\mu T_{-\lambda}\vac{h},
\end{align*} 
which is uniquely determined by $\pair{h}{h}=1$.
By the outer grading operator $d$ defined by $[d,T_n]=n T_n$,
we find that $S(T_\lambda\vac{h},T_\mu\vac{h})=0$ unless $|\lambda|=|\mu|$.
The determinant 
\begin{align*}
D_n(h,q,t)
\seteq (\det S(T_\lambda\vac{h},T_\mu\vac{h}))_{\lambda,\mu\in\calP_n}
\end{align*}
is called the Kac determinant for $\Vir_{q,t}$.

\begin{dfn}
Deformed Gaiotto state $\vac{G}\in M_h$ is defined to be a 
vector satisfying the condition:
\begin{align*}
T_1\vac{G}=\Lambda^2\vac{G},\quad
T_n\vac{G}=0\ (n\ge2),
\end{align*}
where $\Lambda^2$ is a (non-zero) complex number.
The dual vector $\covac{G}\in M_h^*$ is defined similarly:
\begin{align*}
\covac{G}T_{-1}=\Lambda^2\covac{G},\quad
\covac{G}T_n=0\ (n\le-2).
\end{align*}
\end{dfn}

\begin{lem}
If $q$, $t$ and $h$ are  generic complex numbers,
then the deformed Giotto state $\vac{G}$ uniquely exists.
\end{lem}
\begin{proof}
This is the consequence of 
the non-vanishing of the Kac determinant for $\Vir_{q,t}$,
which follows from the factorized formula \cite{BP:1998}:
\begin{align}
\label{eq:kac}
D_n(h,q,t)=C_n 
 \prod_{\substack{r,s\in\bbZ_{\ge1}\\ r s\le n}}(h^2-h_{r,s}^2)^{p(n-r s)}
 \left(\dfrac{(1-q^r)(1-t^{-r})}{q^r+t^r}\right)^{p(n-r s)}.
\end{align}
Here $h_{r,s}\seteq t^{r/2}q^{-s/2}+t^{-r/2}q^{s/2}$, and 
$C_n$ is a constant independent of $q$, $t$ and $h$.
\end{proof}

\subsection{Conjecture on the inner product of the deformed Gaiotto state}

In \cite{AY:2010}, Awata and Yamada conjectured that  
the inner product $\pair{G}{G}$ is equal to the five-dimensional pure $\SU(2)$ 
Nekrasov partition function.
To express $\pair{G}{G}$, it is convenient to introduce the parameter $Q$ by 
\begin{align*}
h=Q^{1/2}+Q^{-1/2}.
\end{align*}
From the Kac determinant formula \eqref{eq:kac},
$\pair{G}{G}$ is a rational function of $q,t$ and $h^2$.
Thus it is a rational function of $q,t$ and $Q$.

We propose a recursive formula on the inner product $\pair{G}{G}$ 
of the deformed Gaiotto state $\vac{G}$.
\begin{cnj}\label{cnj:main}
Assume that $q$ and $t$ are generic complex numbers and that $\Lambda$ is a 
non-zero complex number.
Let $\vac{G}$ be the deformed Gaiotto state of $\Vir_{q,t}$ (see \S \ref{subsec:dGs} for the definition).
Then $\pair{G}{G}$ is of the form 
$F(\Lambda,Q,q,t)=\sum_{n=0}^\infty (\Lambda^{4}t/q)^{n} F_n(Q,q,t)$,
where $F_n(Q,q,t)$ is a rational function 
in terms of the variables $Q,q,t$. 
Moreover it satisfies the following recursive relation :
\begin{align}
\label{eq:rec}
F_n(Q,q,t)
=\delta_{n,0}+\sum_{\substack{r,s\in\bbZ, \\ 1\le r s\le n}}
 \dfrac{G(r,s;q,t)F_{n-r s}(q^r t^{s},q,t)}{Q-q^r t^{-s}}.
\end{align}
Here we introduced the rational function
\begin{align}
\label{eq:g}
G(r,s;q,t)\seteq -\sgn(r) q^r t^{-s} 
\prod_{\substack{-|r|\le i\le|r|-1, \\ -|s|\le j\le|s|-1, \\ (i,j)\neq (0,0)}}
\dfrac{1}{1-q^i t^{-j}}
\end{align}
with $\sgn(r)=1$ if $r>0$ and $\sgn(r)=-1$ if $r<0$.
\end{cnj}

\begin{rmk}
(1)
Recall that in the limit $q=t^\beta$, $t\to 1$, 
the deformed algebra $\Vir_{q,t}$ reduces to the Virasoro algebra 
with the central charge  $c=1-6(\sqrt{\beta}-1/\sqrt{\beta})^2$.
In this limit the deformed Gaiotto state $\vac{G}$ reduces to the 
Gaiotto state $\vac{\Delta,\Lambda^2}$ 
for the $N_f=0$ case \cite[(2.3)]{G:2008}, 
and our conjecture gives the formula appearing 
in \cite[(4.11)]{P:2009} and \cite[(20)]{HJS:2010b}.

(2)
Let us write $S_{q,t}^{(n)}$ the matrix of size $p(n)$ 
defined to be 
\begin{align*}
S_{q,t}^{(n)}\seteq (S(T_\lambda\vac{h},T_\mu\vac{h}))_{\lambda,\mu\in\calP_n}.
\end{align*} 
Note that we have 
\begin{align*}
\pair{G}{G}=\big(S_{q,t}^{(n)}\big)^{-1}(1^n,1^n),
\end{align*}
where the right hand side is the $(1^n,1^n)$-element of 
the inverse matrix of $S_{q,t}^{(n)}$.
Our conjecture implies that 
all the poles of this element with respect to the variable $Q$ are simple.
It is worth noting that in the Virasoro case 
this simple-pole phenomenon is a non-trivial fact. 
By \cite{MMM:2009} the inner product 
$\pair{\Delta,\Lambda^2}{\Delta,\Lambda^2}$ is equal to 
$\sum_{n}\Lambda^{4n}Q_{\Delta}^{-1}(1^n,1^n)$,
where $Q_{\Delta}^{-1}(1^n,1^n)$ is the $(1^n,1^n)$-element 
of the inverse matrix of the contravariant form $Q_{\Delta}$  
of the Virasoro algebra.
In the mathematical literature,
the simple-pole phenomenon of the inverse contravariant form is proved 
in  \cite{B:2003} for the Virasoro case and 
in \cite{O:1992} for the finite Lie algebra case.
See also \cite[Chapter 5]{H:2008:book}.
\end{rmk}

\section{Integral expression of the $K$-theoretic Nekrasov partition function}\label{sect:Nek}

\subsection{Combinatorial definition of the $K$-theoretic Nekrasov partition}
Recall that the instanton part of the $K$-theoretic Nekrasov partition 
function was defined as the integration in the equivariant $K$-theory on the 
framed moduli space of torsion free sheaves on $\bbP^2$ 
in \cite{N:2003,NY:2005}. 
By the Atiyah-Bott-Lefshetz localization theorem 
it becomes a summation over the fixed point contribution. 
These fixed points are parametrized by $r$-tuples of Young diagrams, 
and one obtains a combinatorial form of the partition function.

Following the notation in \cite{AY:2010}, 
we introduce the $K$-theoretic Nekrasov partition function $Z(\Lambda,Q,q,t)$ 
as follows.
\begin{align}
\label{eq:nek}
&Z(\Lambda,Q,q,t)\seteq \sum_{\lambda,\mu\in\calP}
 (\Lambda^4 t/q)^{|\lambda|+|\mu|} Z_{\lambda, \mu}(Q,q,t),
\\
\label{eq:z:lm}
&Z_{\lambda,\mu}(Q,q,t)\seteq
 \dfrac{1}
       {N_{\lambda,\lambda}(1)N_{\mu,\mu}(1)
        N_{\lambda,\mu}(Q)N_{\mu,\lambda}(Q^{-1})},
\\
\label{eq:N}
&N_{\lambda,\mu}(Q)\seteq
 \prod_{(i,j)\in\mu}    (1-Q q^{\lambda_i-j}t^{\mu_j^\vee-i+1})
 \prod_{(i,j)\in\lambda}(1-Q q^{-\mu_i+j-1}t^{-\lambda_j^\vee+i}).
\end{align}

For the convenience of the discussion let us introduce the graded part 
$Z_n(Q,q,t)$ of $Z(\Lambda,Q,q,t)$ by 
\begin{align*}
Z_n(Q,q,t)\seteq \sum_{\substack{\lambda,\mu\in\calP, \\ |\lambda|+|\mu|=n}}
 Z_{\lambda, \mu}(Q,q,t).
\end{align*}
Then we have 
$Z(\Lambda,Q,q,t)=\sum_{n=0}^\infty (\Lambda^{4}t/q)^{n}Z_n(Q,q,t)$.

Note also that $Z_{\lambda,\mu}$ enjoys the following relation :
\begin{align}
\label{eq:Z:duality}
Z_{\lambda,\mu}(Q,q,t)=Z_{\mu,\lambda}(Q^{-1},q,t).
\end{align}

\subsection{Integral expression}
We have the next integral expression of the $K$-theoretic Nekrasov partition 
function. 
Note that the four-dimensional counterpart appeared in \cite{N:2003} and was investigated in \cite{FL:2010,HJS:2010b}. 

\begin{prp}\label{prp:int}
Let $Q,q,t$ be generic complex numbers.
Then the following integral formula holds.
\begin{align}
\label{eq:int}
Z_n(Q,q,t)
=\dfrac{1}{n!}\Bigg(\dfrac{1-q t^{-1}}{(1-q)(1-t^{-1})}\Bigg)^n
 \oint_{C_n}\dfrac{d x_n}{2\pi\sqrt{-1}}\cdots
 \oint_{C_1}\dfrac{d x_1}{2\pi\sqrt{-1}} 
 \Bigg(\\
\nonumber
 \prod_{k=1}^n P(x_k;Q^{1/2},q^{-1} t) \prod_{1\le i<j\le n} \omega(x_j/x_i;q,t^{-1},q^{-1} t)\Bigg),
\end{align}
Here we introduced rational functions
\begin{align*}
&P(x;a,p)\seteq\dfrac{x}{(x-a)(x-a^{-1})(x-p a)(x-p a^{-1})},
\\
&\omega(y;q_1,q_2,q_3)\seteq
  \dfrac{(y-1)^2(y-q_3)(y-q_3^{-1})}{(y-q_1)(y-q_1^{-1})(y-q_2)(y-q_2^{-1})}.
\end{align*}
and the contour $C_k$ is chosen so that it surrounds only the poles at 
$Q^{1/2}$, $Q^{-1/2}$, $x_j q$ and $x_j t^{-1}$ for $j\neq k$.
\end{prp}

Before starting the proof, we show some examples of the integration.
See \cite[\S 2]{FL:2010} for the four-dimensional counterparts.

\begin{eg}
The case $n=0$ holds trivially.
Let us denote this situation by the pair of empty 
partitions $(\emptyset,\emptyset)$

For the case $n=1$, note that the poles with respect to the variable $x_1$ 
are at $Q^{1/2}$ and $Q^{-1/2}$. 
If the pole $Q^{1/2}$ is chosen, then let us associate the pair of partitions
$((1),\emptyset)$.
In this situation the integral is calculated as 
\begin{align*}
&\dfrac{1-q t^{-1}}{(1-q)(1-t^{-1})}
\dfrac{Q^{1/2}}{(Q^{1/2}-Q^{-1/2})(Q^{1/2}-q^{-1} t Q^{1/2})(Q^{1/2}-q^{-1}t Q^{-1/2})}\\
&=\dfrac{1}{(t-1)(q^{-1}-1)(1-Q)(1-q^{-1}t Q^{-1})},
\end{align*}
which is nothing but $Z_{(1),\emptyset}(Q,q,t)$.
Similarly for the choice $Q^{-1/2}$ we associate the pair of Young diagrams
$(\emptyset,(1))$.
The corresponding integral yields $Z_{\emptyset,(1)}(Q,q,t)$.
Thus we have shown that the right hand side of \eqref{eq:int} is $Z_1(Q,q,t)$.

In the case $n=2$, the integral over the variable $x_1$ surrounds two poles 
$Q^{1/2}$ and $Q^{-1/2}$.
Let us choose $Q^{1/2}$ and associate the pair $((1),\emptyset)$.
Then the integral becomes (excluding the factor $(2!)^{-1}$)
\begin{align*}
&\dfrac{(1-q^{-1}t)}{(q^{-1}-1)^2(t-1)^2}\cdot \dfrac{Q^2}{(Q-1)(Q-q^{-1}t)}
\\
&\times
\dfrac{(x_2-q t^{-1}Q^{1/2})(x_2-Q^{1/2})x_2}
      {(x_2-Q^{-1/2})(x_2-q Q^{1/2})(x_2-q^{-1}Q^{1/2})}
\\
&\times
\dfrac{1}{(x_2-t^{-1}Q^{1/2})(x_2-t Q^{1/2})(x_2-t q^{-1}Q^{-1/2})}.
\end{align*}
By the condition of the contour $C_2$,
the choices of the poles with respect to $x_2$ are 
$Q^{-1/2}$, $q Q^{1/2}$, $t^{-1} Q^{1/2}$.
Note that the term $x_2-Q^{1/2}$ cancels with the numerator $x_2/x_1-1$ 
in the function $\omega$, 
so that the pole does not exist at $x_2=Q$.
To each case we associate the pair of partitions $((1),(1))$, 
$((2),\emptyset)$, $((1,1),\emptyset)$.
One can show that each integral coincides with $Z_{(1),(1)}$, 
$Z_{(2),\emptyset}$, $Z_{(1,1),\emptyset}$ by direct calculations.
If we choose $Q^{-1/2}$ instead, then the possible choice of the poles 
with respect to $x_2$ are $Q^{1/2}$, $q Q^{-1/2}$ and $t^{-1}Q^{-1/2}$.
To each case we associate the pair of partitions $((1),(1))$, 
$(\emptyset,(2))$, $(\emptyset,(1,1))$.
One can show that each integral coincides with $Z_{(1),(1)}$, 
$Z_{\emptyset,(2)}$, $Z_{\emptyset,(1,1)}$.
We also have the choices 
$(x_1,x_2)=(q^{1/2} Q^{1/2}, Q^{1/2})$, $(t^{-1/2} Q^{1/2}, Q^{1/2})$, 
$(q^{1/2} Q^{-1/2}, Q^{-1/2})$, $(t^{-1/2} Q^{1/2}, Q^{1/2})$.
They yield $Z_{(2),\emptyset}$, $Z_{(1,1),\emptyset}$, $Z_{\emptyset,(2)}$, $Z_{\emptyset,(1,1)}$ respectively.
Noting the combinatorial factor $1/2!$ in the integral expression,
we find that in the case $n=2$ 
the right hand side of \eqref{eq:int} is $Z_2(Q,q,t)$.
\end{eg}

Now we start the proof.
\begin{proof}[{Proof of Proposition \ref{prp:int}}] 
Let us assume $n\ge 2$.
By the choice of contours, 
the poles of the integral are located at 
\begin{align}
\label{eq:specialization}
 \{x_1,x_2,\ldots,x_n\}
=&\{Q^{1/2} q^{j-1} t^{-(i-1)} 
    \mid 1\le i\le\ell(\lambda), 1\le j \le \lambda_i \}
\\ 
\nonumber
 &\bigcup 
  \{Q^{-1/2} q^{i-1} t^{-(j-1)} \mid 1\le i\le\ell(\mu), 1\le j \le \mu_i \}
\end{align}
with some pair of partitions $(\lambda,\mu)$ such that $|\lambda|+|\mu|=n$.
Let us associate this pair $(\lambda,\mu)$ 
to the choice of poles \eqref{eq:specialization}.

We will show the residue at the pole parametrized by $(\lambda,\mu)$ 
coincides with $Z_{\lambda,\mu}$.
The proof will be done by the induction on $n=|\lambda|+|\mu|$.
The case $\lambda=\mu=\emptyset$ is trivial.
For the general $\lambda$ and $\mu$, 
let us set $L\seteq\ell(\lambda)$, $\ell\seteq\ell(\mu)$, 
$B\seteq\lambda_\ell$, $b\seteq\mu_\ell$, $m\seteq|\lambda|$.
We may fix the ordering in \eqref{eq:specialization} such that 
\begin{align*}
\begin{array}{ll}
 x_{i+\sum_{a=1}^{j-1}\lambda_a}=Q^{1/2} q^{j-1} t^{-(i-1)}
&(1\le i\le  \ell(\lambda)=L,\ 1\le j\le \lambda_i),\\
 x_{m+j+\sum_{a=1}^{j-1}\mu_a}=Q^{-1/2} q^{j-1} t^{-(i-1)}
&(1\le i\le  \ell(\mu)=\ell, \ 1\le j\le \mu_i).
\end{array}
\end{align*}
Then we can remove the factor $(n!)^{-1}$ from the integral.
Let us define 
\begin{align}
\nonumber
&I_{\lambda,\mu}(Q,q,t)\seteq
\Big[\dfrac{1-q t^{-1}}{(1-q)(1-t^{-1})}\Big]^n\times
\\
\nonumber
&\Res_{x_{n}=Q^{-1/2}q^{b-1}t^{-(\ell-1)}}
 \cdots
 \Res_{x_{m+j+\sum_{a=1}^{i-1}\mu_a}=Q^{-1/2} q^{j-1} t^{-(i-1)}}
 \cdots
 \Res_{x_{m+1}=Q^{-1/2}}
\\
\nonumber
&\cdot
 \Res_{x_{m}=Q^{1/2}q^{B-1}t^{-(L-1)}}
 \cdots
 \Res_{x_{j+\sum_{a=1}^{i-1}\lambda_a}=Q^{1/2} q^{j-1} t^{-(i-1)}}
 \cdots
 \Res_{x_{1}=Q^{1/2}}
 \Big[
\\
\label{eq:prp:int:I}
&\prod_{k=1}^n P(x_k;Q^{1/2}) \prod_{1\le i<j\le n} \omega(x_j/x_i)\Big].
\end{align}
Here and afterwards in this proof we use abbreviations
\begin{align*}
P(x;Q)\seteq P(x;Q,q^{-1} t), \quad 
\omega(y)\seteq \omega(y;q,t^{-1},q^{-1} t).
\end{align*}
Then $I_{\lambda,\mu}(Q,q,t)$ satisfies the relation
\begin{align}
\label{eq:prp:int:duality}
I_{\lambda,\mu}(Q,q,t)=I_{\mu,\lambda}(Q^{-1},q,t),
\end{align}
which follows from the definition of the rational functions $P$ and $\omega$.
Note that this is the same as \eqref{eq:Z:duality}.

Now what we must proof is 
\begin{align}\label{eq:prp:int:csq}
Z_{\lambda,\mu}(Q,q,t)=I_{\lambda,\mu}(Q,q,t).
\end{align}
Let us also define $\widetilde{\mu}$ to be the partition made from $\mu$ 
by moving away the last one box. 
In other words we have
\begin{align*}
\widetilde{\mu}\seteq(\mu_1,\ldots,\mu_{\ell-1},b-1).
\end{align*}
We will show that 
\begin{align}
\label{eq:prp:int:rat}
Z_{\lambda,\mu}(Q,q,t)/Z_{\lambda,\widetilde{\mu}}(Q,q,t)
=I_{\lambda,\mu}(Q,q,t)/I_{\lambda,\widetilde{\mu}}(Q,q,t).
\end{align}
Then by repeating the use of \eqref{eq:prp:int:rat},
the desired consequence \eqref{eq:prp:int:csq} reduces to 
$Z_{\lambda,\emptyset}(Q,q,t)=I_{\lambda,\emptyset}(Q,q,t)$,
and by the relations \eqref{eq:Z:duality} and \eqref{eq:prp:int:duality},
it reduces to $Z_{\emptyset,\lambda}(Q,q,t)=I_{\emptyset,\lambda}(Q,q,t)$,
which again by \eqref{eq:prp:int:rat} reduces to the trivial case 
$Z_{\emptyset,\emptyset}(Q,q,t)=I_{\emptyset,\emptyset}(Q,q,t)$.

For the proof of \eqref{eq:prp:int:rat}, we set
\begin{align}
\label{eq:prp:int:cnv}
&\lambda=(\overbrace{m_1,\cdots,m_1}^{n_1},\overbrace{m_2,\cdots,m_2}^{n_2},
          \cdots,\overbrace{m_r,\cdots,m_r}^{n_r}),
\\
\nonumber
&\mu=(\overbrace{\widetilde{m}_1,\cdots,\widetilde{m}_1}^{\widetilde{n}_1},
      \overbrace{\widetilde{m}_2,\cdots,\widetilde{m}_2}^{\widetilde{n}_2},
      \cdots,
      \overbrace{\widetilde{m}_r,\cdots,\widetilde{m}_s}^{\widetilde{n}_s}).
\end{align}
Thus we have $L=n_1+\cdots+n_r$, 
$\ell=\widetilde{n}_1+\cdots+\widetilde{n}_s$, 
$B=m_r$, $b=\widetilde{m}_s$.
We also define the integer $k$ by the condition 
\begin{align*}
n_1+\cdots+n_{k-1}<\ell\le n_1+\cdots+n_k
\end{align*}
if $\ell \le L$. Otherwise we set $k\seteq r+1$.
In the following calculation we set 
$n_0=\widetilde{n}_0=m_{r+1}=\widetilde{m}_{s+1}\seteq0$.

First we treat $Z_{\lambda,\mu}/Z_{\lambda,\widetilde{\mu}}$.
By \eqref{eq:N} and the direct calculation we have
\begin{align*}
&
\dfrac{N_{\lambda,\mu}(Q,q,t)}{N_{\lambda,\widetilde{\mu}}(Q,q,t)}  
\\
&=(1-Q q^{a_\lambda(b,\ell)}t)
 \Big[
 \prod_{\substack{\square\in\lambda \\ i(\square)=\ell}}
 \dfrac{1-Q q^{-a_{\widetilde{\mu}}(\square)-2}t^{-l_\lambda(\square)}}
       {1-Q q^{-a_{\widetilde{\mu}}(\square)-1}t^{-l_\lambda(\square)}}
 \Big]
\\
&\times
 \Big[
 \prod_{\substack{\square\in\mu \\ j(\square)=b}}
 \dfrac{1-Q q^{a_{\widetilde{\mu}}(\square)}t^{-l_\lambda(\square)+2}}
       {1-Q q^{a_{\widetilde{\mu}}(\square)}t^{-l_\lambda(\square)+1}}
 \Big]
\\
&=(1-Q^{a_\lambda(b,\ell)}t)
  \Big[
  \prod_{i=k}^{r}
  \dfrac{1-Q q^{m_i-b}    t^{\ell-(n_1+\cdots+n_i)}}
        {1-Q q^{m_{i+1}-b}t^{\ell-(n_1+\cdots+n_i)}}
  \Big]
\\
&\times
  \Big[
  \prod_{i=1}^{k-1}
  \dfrac{1-Q q^{m_i-b} t^{\ell-(n_1+\cdots+n_{i-1})}}
        {1-Q q^{m_i-b} t^{\ell-(n_1+\cdots+n_i)}}
  \Big]
  \dfrac{1-Q q^{m_k-b} t^{\ell-(n_1+\cdots+n_{k-1})}}{1-Q q^{m_k-b}t}
\\
&=(1-Q q^{-b}t^{\ell-L})
  \prod_{i=1}^{r}
  \dfrac{1-Q q^{m_i-b} t^{\ell-(n_1+\cdots+n_{i-1})}}
        {1-Q q^{m_i-b} t^{\ell-(n_1+\cdots+n_{i})}}.
\end{align*}
In the second equality we used the convention \eqref{eq:prp:int:cnv}.
Similar treatments yield
\begin{align*}
\dfrac{N_{\mu,\lambda}(Q^{-1},q,t)}{N_{\widetilde{\mu},\lambda}(Q^{-1},q,t)}
=&(1-Q^{-1}q^{b-1}t^{L-\ell+1})
\\
&\times
  \prod_{i=1}^{r}
  \dfrac{1-Q^{-1} q^{b-1-m_i} t^{(n_1+\cdots+n_{i-1})-\ell+1}}
        {1-Q^{-1} q^{b-1-m_i} t^{(n_1+\cdots+n_{i})  -\ell+1}},
\\
\dfrac{N_{\mu,\mu}(1,q,t)}{N_{\widetilde{\mu},\widetilde{\mu}}(1,q,t)}
=&(1-q^{-b})(1-q^{b-1}t)
\Big[\prod_{i=1}^{s} 
\dfrac{1-q^{b-\widetilde{m}_i-1}
       t^{(\widetilde{n}_1+\cdots+\widetilde{n}_{i-1})-\ell+1}}
      {1-q^{b-\widetilde{m}_i-1}
       t^{(\widetilde{n}_1+\cdots+\widetilde{n}_{i})-\ell+1}}
\Big]
\\
&\times
\Big[\prod_{i=1}^{s} 
\dfrac{1-q^{\widetilde{m}_i-b}
       t^{\ell-(\widetilde{n}_1+\cdots+\widetilde{n}_{i-1})}}
      {1-q^{\widetilde{m}_i-b}
       t^{\ell-(\widetilde{n}_1+\cdots+\widetilde{n}_{i})}}
\Big]
\end{align*}
Thus by \eqref{eq:z:lm} and by some elementary calculation we have
\begin{align}
\nonumber
&\dfrac{Z_{\lambda,\mu}(Q,q,t)}{Z_{\lambda,\widetilde{\mu}}(Q,q,t)}
=\dfrac{Q q^{2(1-b)}t^{2(\ell-1)}}
 {(1-Q^{-1}q^{b-1}t^{L-\ell+1})(1-Q q^{a_\lambda(b,\ell)}t)
  (1-q^{-b})(1-q^{b-1}t)}
\\
\nonumber
&\times
 \Big[
 \prod_{i=1}^{r}
 \dfrac{1-Q q^{m_i-b} t^{\ell-(n_1+\cdots+n_{i})}}
       {1-Q q^{m_i-b} t^{\ell-(n_1+\cdots+n_{i-1})}}
 \Big]
 \Big[
  \prod_{i=1}^{r}
  \dfrac{1-Q q^{m_i-b+1} t^{\ell-(n_1+\cdots+n_{i})  -1}}
        {1-Q q^{m_i-b+1} t^{\ell-(n_1+\cdots+n_{i-1})-1}}        
 \Big]
\\
\nonumber
&\times
\Big[\prod_{i=1}^{s} 
\dfrac{1-q^{\widetilde{m}_i-b+1}
       t^{\ell-(\widetilde{n}_1+\cdots+\widetilde{n}_{i})-1}}
      {1-q^{\widetilde{m}_i-b+1}
       t^{\ell-(\widetilde{n}_1+\cdots+\widetilde{n}_{i-1})-1}}
\Big]
\Big[\prod_{i=1}^{s-1} 
\dfrac{1-q^{\widetilde{m}_i-b}
       t^{\ell-(\widetilde{n}_1+\cdots+\widetilde{n}_{i})}}
      {1-q^{\widetilde{m}_i-b}
       t^{\ell-(\widetilde{n}_1+\cdots+\widetilde{n}_{i-1})}}
\Big]
\\
\label{eq:prp:int:prp:z}
&\times
 \dfrac{1-t^{1}}
       {1-t^{\ell-(\widetilde{n}_1+\cdots+\widetilde{n}_{s-1})}}.
\end{align}

Next we calculate $I_{\lambda,\mu}/I_{\lambda,\widetilde{\mu}}$.
By the definition \eqref{eq:prp:int:I} we have
\begin{align}
\nonumber
&\dfrac{I_{\lambda,\mu}(Q,q,t)}{I_{\lambda,\widetilde{\mu}}(Q,q,t)}
=\dfrac{1-q t^{-1}}{(1-q)(1-t^{-1})}\times
\\
\nonumber
&P(Q^{-1/2}q^{b-1}t^{-\ell-1};Q^{1/2},q^{-1}t)\times
Q^{-1/2}q^{b-1}t^{-\ell-1}\times
\\
\nonumber
&
\Big[\prod_{(i,j)\in\lambda}
\omega(\dfrac{Q^{-1/2}q^{b-1}t^{-(\ell-1)}}{Q^{1/2}q^{j-1}t^{-(i-1)}}
      ;q,t^{-1},q^{-1}t)\Big]\times
\\
&\Big[\prod_{\substack{(i,j)\in\widetilde{\mu} \\ (i,j)\neq (\ell,b-1)}}
\omega(\dfrac{Q^{-1/2}q^{b-1}t^{-(\ell-1)}}{Q^{-1/2}q^{j-1}t^{-(i-1)}}
      ;q,t^{-1},q^{-1}t)\Big]
\dfrac{(q-1)^2(q-q t^{-1})(q-q^{-1}t)}{(q-q^{-1})(q-t)(q-t^{-1})}.
\label{eq:prp:int:I:prp}
\end{align}
An elementary calculation yields
\begin{align*}
&\prod_{(i,j)\in\lambda}
\omega(\dfrac{Q^{-1/2}q^{b-1}t^{-(\ell-1)}}{Q^{1/2}q^{j-1}t^{-(i-1)}}
      ;q,t^{-1},q^{-1}t)
\\
&=\dfrac{1-Q q^{1-b}t^{\ell-1}}{1-Q^{1-b}t^{\ell-L-1}}
  \dfrac{1-Q q^{-b}t^{\ell}}{1-Q^{-b}t^{\ell-L}}
\\
&\times
\Big[
 \prod_{i=1}^{L}
 \dfrac{1-Q q^{\lambda_i+1-b} t^{\ell-i-1}}{1-Q q^{\lambda_i+1-b} t^{\ell-i}}
\Big]
\Big[
 \prod_{i=1}^{L}
 \dfrac{1-Q q^{\lambda_i-b} t^{\ell-i}}{1-Q q^{\lambda_i-b} t^{\ell-i+1}}
\Big]
\\
&=\dfrac{1-Q q^{1-b}t^{\ell-1}}{1-Q^{1-b}t^{\ell-L-1}}
  \dfrac{1-Q q^{-b}t^{\ell}}{1-Q^{-b}t^{\ell-L}}
\\
&
\times
 \Big[
 \prod_{i=1}^{r}
 \dfrac{1-Q q^{m_i-b} t^{\ell-(n_1+\cdots+n_{i})}}
       {1-Q q^{m_i-b} t^{\ell-(n_1+\cdots+n_{i-1})}}
 \Big]
 \Big[
  \prod_{i=1}^{r}
  \dfrac{1-Q q^{m_i-b+1} t^{\ell-(n_1+\cdots+n_{i})  -1}}
        {1-Q q^{m_i-b+1} t^{\ell-(n_1+\cdots+n_{i-1})-1}}        
 \Big],
\end{align*}
where at the second equality we used the convention \eqref{eq:prp:int:cnv}.
A similar consideration gives
\begin{align*}
&\Big[\prod_{\substack{(i,j)\in\widetilde{\mu} \\ (i,j)\neq (\ell,b-1)}}
\omega(\dfrac{Q^{-1/2}q^{b-1}t^{-(\ell-1)}}{Q^{-1/2}q^{j-1}t^{-(i-1)}}
      ;q,t^{-1},q^{-1}t)\Big]
\dfrac{(q-1)^2(q-q t^{-1})(q-q^{-1}t)}{(q-q^{-1})(q-t)(q-t^{-1})}
\\
&=
\dfrac{1-q^{1-b}t^{\ell-1}}{1-q^{1-b}t^{-1}}
\dfrac{1-q^{-b}t^{\ell}}{1-q^{-b}}
\\
&
\times
\Big[\prod_{i=1}^{s} 
\dfrac{1-q^{\widetilde{m}_i-b+1}
       t^{\ell-(\widetilde{n}_1+\cdots+\widetilde{n}_{i})-1}}
      {1-q^{\widetilde{m}_i-b+1}
       t^{\ell-(\widetilde{n}_1+\cdots+\widetilde{n}_{i-1})-1}}
\Big]
\Big[\prod_{i=1}^{s-1} 
\dfrac{1-q^{\widetilde{m}_i-b}
       t^{\ell-(\widetilde{n}_1+\cdots+\widetilde{n}_{i})}}
      {1-q^{\widetilde{m}_i-b}
       t^{\ell-(\widetilde{n}_1+\cdots+\widetilde{n}_{i-1})}}
\Big]
\\
&\times
\dfrac{1-t}
      {1-t^{\ell-(\widetilde{n}_1+\cdots+\widetilde{n}_{s-1})}}.
\end{align*}
Substituting these terms, 
one finds that \eqref{eq:prp:int:I:prp} is equal to \eqref{eq:prp:int:prp:z}.
This is the desired consequence.
\end{proof}

\begin{rmk}
In the above integral expression one may find a reminiscence 
of the Feigin-Odesskii algebra with three parameters,
which is investigated 
in \cite{FHHSY:2009,FT:2009,SV:2009}. 
Note also that the specialization \eqref{eq:specialization} 
also appeared in the Gordon filtration introduced in \cite{FHHSY:2009}.
\end{rmk}

\section{The recursion relation of $Z(\Lambda,Q,q,t)$}\label{sect:rec}

The main theorem of this paper is 
\begin{thm}\label{thm:main}
The five-dimensional Nekrasov partition function 
\[Z(\Lambda,Q,q,t)=\sum_{n=0}^\infty (\Lambda^{4}t/q)^{n}Z_n(Q,q,t)\]
satisfies the same recursive formula \eqref{eq:rec} 
as in Conjecture \ref{cnj:main}:
\begin{align*}
Z_n(Q,q,t)
=\delta_{n,0}+\sum_{\substack{r,s\in\bbZ, \\ 1\le r s\le n}}
 \dfrac{G(r,s;q,t)Z_{n-r s}(q^r t^{s},q,t)}{Q-q^r t^{-s}}.
\end{align*}
\end{thm}

The proof is divided in two steps 
(Propositions \ref{prp:rec:1} and \ref{prp:rec:2}).
We first show that each pole of $Z_n$ is simple when $Z_n$ is viewed as a 
rational function of the variable $Q$ .

\begin{prp}\label{prp:rec:1}
As a rational function of the variable $Q$,
$Z_n(Q,q,t)$ has poles at $Q=q^r t^{-s}$ ($r,s\in\bbZ$, $1\le r s \le n$).
Moreover each pole is simple. 
\end{prp}

Before starting the proof of this proposition, we prepare the next lemma, 
which is inspired by \cite[\S A]{FL:2010} and \cite[Appendix]{HJS:2010b}.

\begin{lem}\label{lem:I}
Let $f(x_1,\ldots,x_n)$ be a symmetric polynomial function 
of all its arguments.
Assume that the integral
\begin{align}
\label{eq:I}
&I_n(Q,q,t)=
\oint_{C_n}\dfrac{d x_n}{2\pi\sqrt{-1}}\cdots
 \oint_{C_1}\dfrac{d x_1}{2\pi\sqrt{-1}} 
 \Bigg(
 \dfrac{f(x_1,\ldots,x_n)}{ \prod_{k=1}^n (x_k-Q^{1/2})(x_k-Q^{-1/2})}
\\
\nonumber
&\times
 \prod_{1\le i<j\le n}
 \dfrac{(x_j/x_i-1)^2(x_j/x_i-q t^{-1})(x_j/x_i-q^{-1} t)}
        {(x_j/x_i-q)(x_j/x_i-q^{-1})(x_j/x_i-t)(x_j/x_i-t^{-1})}
 \Bigg)
\end{align}
converges, 
where the contour $C_k$ is chosen so that it surrounds only the poles at 
$Q^{1/2}$, $Q^{-1/2}$, $x_j q$ and $x_j t^{-1}$ for $j \neq k$,
Then $I_n(Q,q,t)$ is a holomorphic function of $Q$.
\end{lem}

\begin{proof}
Let us integrate over the variable $x_n$ first.
By the choice of $C_n$, 
the poles at $x_n=Q^{\pm1/2}$, $x_n=x_j q$ and $x_n=x_j t^{-1}$ ($j<n$)
are inside the contour,
while the poles at $x_n=x_j q^{-1}$ and $x_n=x_j t$ are outside.

If we vary  $Q$ then some of these poles move.
When two of the poles coincide, 
a higher order singularity of the integrand may occur.
If this happens for the poles on the same side of the contour,
then one can move the contour away from the colliding pair 
and the integral contains no singularity.
If the pair approach from the opposite sides of the contour,
then the integral becomes singular.

It is convenient to express the contour $C_n$ as 
$C_n=C_n^{a}-\sum_{1\le k\le n-1}C_{n,k}^{d}$,
where $C_n^{a}$ is the contour surrounding all possible poles of $x_n$, 
and $C_{n,k}^{d}$ is the contour 
surrounding the point $x_k q^{-1}$ and $x_k t$.
By the observation above, 
the integral over $C_n^{a}$ does not contribute to the singularity.
Thus we only have to consider the integrals over $C_{n,k}^{d}$.
The factors depending on $x_n$ are expressed as follows:
\begin{align*}
&J_n\seteq
\sum_{1\le k\le n-1}
\oint_{-C_{n,k}^d}
\dfrac{d x_n}{2\pi\sqrt{-1}}\Bigg(
\dfrac{f(x_1,\ldots,x_n)}{(x_n-Q^{1/2})(x_n-Q^{-1/2})}
\\
&\times
 \prod_{1\le i\le n-1}
 \dfrac{(x_n/x_i-1)^2(x_n/x_i-q t^{-1})(x_n/x_i-q^{-1}t)}
       {(x_n/x_i-q)(x_n/x_i-q^{-1})(x_n/x_i-t)(x_n/x_i-t^{-1})}
\Bigg)
\\
=&-\sum_{1\le k\le n-1}
\Bigg[\dfrac{f(x_1,\ldots,x_{n-1},x_k q^{-1})}
      {(x_k q^{-1}-Q^{1/2})(x_k q^{-1}-Q^{-1/2})}
\dfrac{(1-q)^2(1-q^2 t^{-1})(1-t)}
      {(1-q^2)(1-q t)(1-q t^{-1})}
\\
&\times
 \prod_{\substack{1\le i\le n-1 \\ i\neq k}}
 \dfrac{(x_k/x_i-q)^2(x_k/x_i-q^2 t^{-1})(x_k/x_i-t)}
       {(x_k/x_i-q^2)(x_k/x_i-1)(x_k/x_i-q t)(x_k/x_i-q t^{-1})}
\\
&+\dfrac{f(x_1,\ldots,x_{n-1},x_k t)}{(x_k t-Q^{1/2})(x_k t-Q^{-1/2})}
  \dfrac{(1-t^{-1})^2(1-q t^{-2})(1-q^{-1})}
        {(1-t^{-2})(1-q t^{-1})(1-q^{-1}t^{-1})}
\\
&\times
 \prod_{\substack{1\le i\le n-1 \\ i\neq k}}
 \dfrac{(x_k/x_i-t^{-1})^2(x_k/x_i-q t^{-2})(x_k/x_i-q^{-1})}
       {(x_k/x_i-q t^{-1})(x_k/x_i-q^{-1}t^{-1})(x_k/x_i-1)(x_k/x_i-t^{-2})}
\Bigg].
\end{align*}
Now we have several apparent poles in the integrand.
Since we have assumed that $f(x_1,\ldots,x_n)$ is symmetric,
we find that $x_k=x_i$ is not a genuine pole.
Moreover two types of poles $x_k=x_i q t$ and $x_k=x_i q^{-1} t^{-1}$
are not poles. 
In fact if we exchange $i$ and $k$ in the second line of $J_n$ 
then the terms containing $x_k=x_i q t$ can be calculated as
\begin{align}
\nonumber
&\dfrac{f(x_1,\ldots,x_{n-1},x_k q^{-1})}
       {(x_k q^{-1}-Q^{1/2})(x_k q^{-1}-Q^{-1/2})}
\dfrac{(1-q)(1-q^2 t^{-1})(1-t)}{(1+q)(1-q t)(1-q t^{-1})}
\\
\nonumber
&\times
 \prod_{\substack{1\le j\le n-1 \\ j\neq k}}
 \dfrac{(x_k/x_j-q)^2(x_k/x_j-q^2 t^{-1})(x_k/x_j-t)}
       {(x_k/x_j-q^2)(x_k/x_j-1)(x_k/x_j-q t)(x_k/x_j-q t^{-1})}
\\
\nonumber
&+\dfrac{f(x_1,\ldots,x_{n-1},x_i t)}{(x_i t-Q^{1/2})(x_i t-Q^{-1/2})}
  \dfrac{(1-t^{-1})(1-q t^{-2})(1-q^{-1})}
        {(1+t^{-1})(1-q t^{-1})(1-q^{-1}t^{-1})}
\\
\nonumber
&\times
 \prod_{\substack{1\le j\le n-1 \\ j\neq k}}
 \dfrac{(x_i/x_j-t^{-1})^2(x_i/x_j-q t^{-2})(x_i/x_j-q^{-1})}
       {(x_i/x_j-q t^{-1})(x_i/x_j-q^{-1}t^{-1})(x_i/x_j-1)(x_i/x_j-t^{-2})}
\\
\label{eq:lem:I:1}
&=\dfrac{(1-q)(1-t)}{(1-q t)(1-q t^{-1})}
  \dfrac{(x_k/x_i-q)(x_k/x_i-t)}{(x_k/x_i-1)(x_k/x_i-q t)} 
  H(x_1,\ldots,x_{n-1},Q,q,t)
\end{align}
where we have introduced
\begin{align*}
&H(x_1,\ldots,x_{n-1},Q,q,t)\seteq
\\
&\dfrac{f(x_1,\ldots,x_{n-1},x_k q^{-1})}
      {(x_k q^{-1}-Q^{1/2})(x_k q^{-1}-Q^{-1/2})}
\dfrac{1-q^2 t^{-1}}{1+q}
\dfrac{(x_k/x_i-q)(x_k/x_i-q^2 t^{-1})}{(x_k/x_i-q^2)(x_k/x_i-q t^{-1})}
\\
&\times
 \prod_{\substack{1\le j\le n-1 \\ j\neq i,k}}
 \dfrac{(x_k/x_j-q)^2(x_k/x_j-q^2 t^{-1})(x_k/x_j-t)}
       {(x_k/x_j-q^2)(x_k/x_j-1)(x_k/x_j-q t)(x_k/x_j-q t^{-1})}
\\
&-\dfrac{f(x_1,\ldots,x_{n-1},x_i t)}{(x_i t-Q^{1/2})(x_i t-Q^{-1/2})}
  \dfrac{1-q t^{-2}}{1+t^{-1}}
  \dfrac{(x_k/x_i-t)(x_k/x_i-q^{-1} t^{2})}{(x_k/x_i-t^2)(x_k/x_i-q^{-1} t}
\\
&\times
  \dfrac{(x_i/x_j-t^{-1})^2(x_i/x_j-q^2 t^{-1})(x_i/x_j-q^{-1})}
        {(x_i/x_j-q t^{-1})(x_i/x_j-1)(x_i/x_j-q^{-1} t^{-1})(x_i/x_j-t^{-2})}.
\end{align*}
Then we can calculate the residue of \eqref{eq:lem:I:1} at $x_k=x_i q t$ as
\begin{align*}
&H(x_1,\ldots,x_{n-1},Q,q,t)|_{x_k=x_i q t}\\
&=f(x_1,\ldots,x_{n-1},x_i t)\\
&\times
 \prod_{\substack{1\le j\le n-1 \\ j\neq i,k}}
 \dfrac{(x_i/x_j-t^{-1})^2(x_i/x_j-q^2 t^{-1})(x_i/x_j-q^{-1})}
        {(x_i/x_j-q t^{-1})(x_i/x_j-1)(x_i/x_j-q^{-1} t^{-1})(x_i/x_j-t^{-2})}
\\
&\times\Bigg(
\dfrac{1-q^2 t^{-1}}{1+q}
\dfrac{(q t-q)(q t-q^2 t^{-1})}{(qt-q^2)(q t-q t^{-1})}
-
\dfrac{1-q t^{-2}}{1+t^{-1}}
\dfrac{(q t-t)(q t-q^{-1} t^{2})}{(q t-t^2)(q t-q^{-1} t)}
\Bigg).
\end{align*}
Since the last line is zero, $J_n$ contains no pole at $x_k=x_i q t$.
The case $x_k=x_i q^{-1} t^{-1}$ is similar and we omit the detail.

Thus after performing the integral over $x_n$ the function $I_n$ 
becomes an integral over $x_1,\ldots,x_{n-1}$ with 
the function $f$ replaced by some holomorphic function.
The parameters $(Q^{\pm1},q,t^{-1})$ are replaced by 
$(Q^{\pm1}q^i t^{-j},q^a,t^{-b})$ with some $i,j,a,b\in\bbZ_{\ge1}$.
At last the integral becomes the form as
\begin{align}
\label{eq:lem:I:2}
&I_n(Q,q,t)=
 \oint_{C_1}\dfrac{d x_1}{2\pi\sqrt{-1}} 
 \dfrac{\widetilde{f}(x_1)}
       {\prod_{\alpha,\beta}(x_1-Q^{1/2}q^\alpha t^{-\beta})
         (x_1-Q^{-1/2}q^\alpha t^{-\beta})}.
\end{align}
Here $\alpha$ and $\beta$ run over a finite set of non-negative integers,
and $\widetilde{f}$ is a polynomial of $x_1$.
We also find that no higher poles occur in this integral, i.e.,
the terms $x_1-Q^{\pm 1/2}q^\alpha t^{-\beta}$ are different from each other.
By the choice of $C_1$ all the poles of $x_1$ are inside the contour,
so that the result of the integral contains no singularity with respect to $Q$.
This is the desired consequence.
\end{proof}

Now we turn to the proof of Proposition \ref{prp:rec:1}
\begin{proof}[{Proof of  Proposition  \ref{prp:rec:1}}]
The integral expression \eqref{eq:int} of $Z_n$ 
is the same form as $I_n$ \eqref{eq:I} except for the term $f(x_1,\ldots,x_n)$ 
which is not holomorphic. 
Thus by the same argument of the proof of Lemma \ref{lem:I}, 
The pole of $Z_n$ as a rational function of $Q$ may only come from the 
collision of poles at $x_i=Q^{\pm1/2}q^{\alpha}t^{-\beta}$ and 
$x_i=Q^{\mp1/2}q^{-1}t$. 
Note that the singularity with respect to $Q$ are all single poles 
by the argument after \eqref{eq:lem:I:2} 
and by the assumption that the parameters are generic.
Then by the original definition \eqref{eq:nek} of $Z_n$, 
the poles of $Q$ are at 
$Q=q^{r}t^{-s}$ ($r,s\in\bbZ$, $1\le r s \le n$).
This is the desired consequence.
\end{proof}

The second step calculates the residue of $Z_n(Q,q,t)$ at $Q=q^r t^{-s}$.
\begin{prp}\label{prp:rec:2}
\begin{align}
\label{eq:prp:rec:2:0}
\Res_{Q=q^r t^{-s}} Z_n(Q,q,t)=G(r,s;q,t) Z_{n-r s}(q^r t^s,q,t).
\end{align}
\end{prp}
\begin{proof}
We assume that $r,s\ge1$. The case $r,s\le -1$ can be treated similarly.

{}From the combinatorial expression \eqref{eq:nek}, 
the pole at $Q=q^r t^{-s}$ only comes from the pair $(\lambda,\mu)$ of 
partitions such that $\mu$ contains (if displayed by the Young diagram) 
the rectangle $r^s$.

The residue of the pole at $Q=q^r t^{-s}$ 
can be evaluated by using the integral form \eqref{eq:int}.
The condition of $(\lambda,\mu)$ corresponds to
taking the residues of the poles with respect to $x_i$ 
at $Q^{-1/2}, Q^{-1/2}q, \ldots, Q^{-1/2}q^r t^{-s}$.
By relabeling the indices of $x_i$, we have
\begin{align*}
&\Res_{Q=q^r t^{-s}} Z_n(Q,q,t)
=\\
&\dfrac{1}{(n-r s)!}\Bigg(\dfrac{1-q t^{-1}}{(1-q)(1-t^{-1})}\Bigg)^{n-r s}
\oint_{C_n}\dfrac{d x_n}{2\pi\sqrt{-1}}\cdots
\oint_{C_{r s+1}}\dfrac{d x_{r s+1}}{2\pi\sqrt{-1}}\Bigg(  
\\
& \prod_{k=r s+1}^n P(x_k;Q^{1/2},q^{-1} t) 
 \prod_{k<l} \omega(x_l/x_k;q,t^{-1},q^{-1} t) K_{r,s}\Bigg).
\end{align*}
Here the indices $i,j,k,l$ runs as $1\le i,j\le r s$ and $rs<k,l\le n$.
The function $K_{r,s}$ is defined to be
\begin{align*}
&K_{r,s}\seteq\Res_{Q=q^r t^{-s}} 
\dfrac{1}{(r s)!}\Bigg(\dfrac{1-q t^{-1}}{(1-q)(1-t^{-1})}\Bigg)^{r s}
\oint_{C_{r s}}\dfrac{d x_{r s}}{2\pi\sqrt{-1}}\cdots
\oint_{C_{1}}\dfrac{d x_{1}}{2\pi\sqrt{-1}}\Bigg(
\\
&\prod_{i,k} \omega(x_k/x_i;q,t^{-1},q^{-1} t)
 \prod_{i=1}^{r s} P(x_i;Q^{1/2},q^{-1} t)
 \prod_{i<j} \omega(x_j/x_i;q,t^{-1},q^{-1} t)
\Bigg)
\\
&=\Res_{Q=q^r t^{-s}} 
 Z_{\emptyset,(r^s)}(Q,q,t)\prod_{k=r s+1}^N
 \prod_{\alpha=1}^r\prod_{\beta=1}^s
 \omega(x_k/(Q^{-1/2} q^\alpha t^{-\beta});q,t^{-1},q^{-1}t).
\end{align*}
Note that
\begin{align}
\label{eq:Po}
\begin{split}
&P(x;Q^{1/2},q^{-1}t)
\prod_{\alpha=1}^r\prod_{\beta=1}^s
 \omega(x/(Q^{-1/2}q^\alpha t^{-\beta});q,t^{-1},q^{-1}t)\\
&=\dfrac{x}{(x-Q^{-1/2} q^r)(x-Q^{-1/2} t^{-s})
 (x-Q^{-1/2} q^{r}q^{-1}t)(x-Q^{-1/2} t^{-s} q^{-1} t)}\\
&\times \dfrac{(x-Q^{-1/2} q^{r} t^{-s})(x-Q^{-1/2} q^{r-1} t^{-s+1})}
       {(x-Q^{1/2})(x-Q^{1/2}q^{-1}t)}.
\end{split}
\end{align}
If $Q=q^r t^{-s}$, then the last line of \eqref{eq:Po} is equal to one 
and we have
\begin{align*}
&P(x;Q^{1/2},q^{-1}t)
\prod_{\alpha=1}^r\prod_{\beta=1}^s
 \omega(x/(Q^{1/2}q^\alpha t^{-\beta});q,t^{-1},q^{-1}t)\Big|_{Q=q^r t^{-s}}
\\
&=\dfrac{x}{(x- q^{r/2}t^{s/2})(x-q^{-r/2} t^{-s/2})
 (x-q^{r/2}t^{s/2} q^{-1}t)(x-q^{-r/2} t^{-s/2} q^{-1} t)}
\\
&=P(x;(q^{r}t^{s})^{1/2},q^{-1}t).
\end{align*}
Thus we obtain
\begin{align}
\label{eq:prp:rec:2:1}
\Res_{Q=q^r t^{-s}} Z_n(Q,q,t)
=Z_{n-r s}(q^{r}t^{s},q,t) \cdot 
 \Res_{Q=q^r t^{-s}} Z_{\emptyset,(r^s)}(Q,q,t).
\end{align}

Comparing \eqref{eq:prp:rec:2:1} and the statement \eqref{eq:prp:rec:2:0},
we find that it is enough to show 
$\Res_{Q=q^r t^{-s}} Z_{\emptyset,(r^s)}(Q,q,t)=G(r,s;q,t)$.
By the direct calculation we have
\begin{align*}
&N_{\emptyset,\emptyset}(1)=1,
\\
&N_{(r^s),(r^s)}(1)
=\prod_{\substack{-r\le i\le -1\\ 0\le j \le s-1}}(1-q^i t^{-j})
 \times 
 \prod_{\substack{0\le i\le r-1\\ -s\le j \le -1}}(1-q^i t^{-j}),
\\
&\Res_{Q=q^r t^{-s}} N_{(\emptyset),(r^s)}(Q)
 =(-q^r t^{-s})
\prod_{\substack{0\le i\le r-1\\ 0\le j \le s-1 \\ (i,j)\neq(0,0)}}
 (1-q^i t^{-j}),
\\
&N_{(r^s),(\emptyset)}(Q^{-1})\Big|_{Q=q^r t^{-s}}
 =\prod_{\substack{-r\le i\le -1\\ -s\le j \le -1}}(1-q^i t^{-j}).
\end{align*}
Multiplying these factors 
and recalling the definition \eqref{eq:g} of $G(r,s;q,t)$, 
we obtain
\begin{align*}
\Res_{Q=q^r t^{-s}} Z_{\emptyset,(r^s)}(Q,q,t)
=-q^r t^{-s}
\prod_{\substack{-r\le i\le r-1\\ -s\le j \le s-1 \\ (i,j)\neq(0,0)}}
(1-q^i t^{-j})=G(r,s;q,t).
\end{align*}
This is the desired consequence.
\end{proof}


\end{document}